\documentclass{amsart}
\usepackage[dvips]{graphicx}

\newtheorem{theorem}{Theorem}
\newtheorem{lemma}[theorem]{Lemma}
\newtheorem{prop}[theorem]{Proposition}

\theoremstyle{definition}

\newtheorem{example}[theorem]{Example}

\theoremstyle{remark}
\newtheorem{remark}[theorem]{Remark}

\numberwithin{equation}{section}

\begin{document}

\title{Rotation number of primitive vector sequences}

\author{Yusuke Suyama}
\address{Department of Mathematics, Graduate School of Science, Osaka City University, 3-3-138 Sugimoto, Sumiyoshi-ku, Osaka 558-8585 JAPAN}
\email{uniformlyconvergent@gmail.com}



\date{}

\dedicatory{}

\begin{abstract}
We give a formula on the rotation number of a sequence of primitive vectors,
which is a generalization of the formula on the rotation number of a unimodular sequence
in \cite{Higashitani and Masuda}.
\end{abstract}

\maketitle

\section{Introduction}

Let $v_1, \ldots, v_d \in \mathbb{Z}^2$ be a sequence of primitive vectors
such that $\varepsilon_i={\rm det}(v_i, v_{i+1}) \ne 0$ for all $i=1, \ldots, d$,
and let $a_i=\varepsilon_{i-1}^{-1}\varepsilon_i^{-1}{\rm det}(v_{i+1}, v_{i-1})$,
where $v_0=v_d$ and $v_{d+1}=v_1$.
The {\it rotation number} of the sequence $v_1, \ldots, v_d$ around the origin is defined by
\begin{equation*}
\frac{1}{2\pi}\sum_{i=1}^d\int_{L_i}\frac{-ydx+xdy}{x^2+y^2},
\end{equation*}
where $L_i$ is the line segment from $v_i$ to $v_{i+1}$.
The sequence is called {\it unimodular} if $|\varepsilon_i|=1$ for all $i=1, \ldots, d$.
Recently A. Higashitani and M. Masuda \cite{Higashitani and Masuda} proved the following:

\begin{theorem}[\cite{Higashitani and Masuda}]\label{Higashitani-Masuda}
The rotation number of a unimodular sequence $v_1, \ldots, v_d$ around the origin is given by
\begin{equation*}
\frac{1}{12}\sum_{i=1}^d(3\varepsilon_i+a_i).
\end{equation*}
\end{theorem}

When $\varepsilon_i=1$ for all $i$ and the rotation number is one,
Theorem \ref{Higashitani-Masuda} is well known and formulated as $3d+\sum_{i=1}^da_i=12$.
It can be proved in an elementary way, but interestingly it can also be proved using toric geometry,
to be more precise, by applying N\"{o}ther's formula to complete non-singular toric varieties of complex dimension two,
see \cite{Fulton}.
When $\varepsilon_i=1$ for all $i$ but the rotation number is not necessarily one,
Theorem \ref{Higashitani-Masuda} was proved in \cite{Masuda} using toric topology.
The proof is a generalization of the proof above using toric geometry.
The original proof of Theorem \ref{Higashitani-Masuda} by Higashitani and Masuda
was a slight modification of the proof in \cite{Masuda}
but then they found an elementary proof.
Another elementary proof of Theorem \ref{Higashitani-Masuda} is given by R. T. Zivaljevic \cite{Zivaljevic}.

Theorem \ref{Higashitani-Masuda} does not hold
when the unimodularity condition is dropped.
In this paper, we give a formula on the rotation number of a (not necessarily unimodular)
sequence of primitive vectors $v_1, \ldots, v_d$ with $\varepsilon_i \ne 0$ for all $i$, see Theorem \ref{main theorem}.
The proof is done by adding primitive vectors in an appropriate way
to the given sequence so that the enlarged sequence is unimodular
and then by applying Theorem \ref{Higashitani-Masuda} to the enlarged unimodular sequence.
This combinatorial process, that is, making the given sequence to a unimodular sequence
by adding primitive vectors corresponds to resolution of singularity by blow-up in geometry, see \cite{Fulton}.

The structure of the paper is as follows:
In Section 2, we state the main theorem and give an example.
In Section 3, we discuss Hirzebruch-Jung continued fractions used in our proof of the main theorem.
In Section 4, we give a proof of the main theorem.

\section{The main theorem}

Let $v_1, \ldots, v_d \in \mathbb{Z}^2$ be a sequence of vectors
such that $\varepsilon_i={\rm det}(v_i, v_{i+1}) \ne 0$ for all $i=1, \ldots, d$.
We define $v_0=v_d$ and $v_{d+1}=v_1$.
We assume that each vector is {\it primitive}, i.e. its components are relatively prime.

\begin{lemma}\label{x_i}
For each $i=1, \ldots, d$,
there exists a unique non-negative integer $x_i<|\varepsilon_i|$ such that
$x_i$ and $|\varepsilon_i|$ are relatively prime and
\begin{equation*}
P_i=(v_i, v_{i+1})\left(\begin{array}{cc}
1 & -x_i \\
0 & |\varepsilon_i|
\end{array}\right)^{-1}
\end{equation*}
is a unimodular matrix.
\end{lemma}

\begin{proof}
Let $v_i=\left(\begin{array}{c}a\\b\end{array}\right)$ and $v_{i+1}=\left(\begin{array}{c}c\\d\end{array}\right)$.
We assume that $\varepsilon_i>0$.
Since $v_i$ is primitive, there exist $p, q \in \mathbb{Z}$ such that $ap+bq=1$.
Then we have
\begin{equation*}
\left(\begin{array}{cc}
p & q \\
-b & a
\end{array}\right)(v_i, v_{i+1})=\left(\begin{array}{cc}
1 & cp+dq \\
0 & |\varepsilon_i|
\end{array}\right).
\end{equation*}
There exists a unique $n \in \mathbb{Z}$ satisfying $-|\varepsilon_i|<cp+dq+n|\varepsilon_i| \leq 0$.
So we put $x_i=-(cp+dq+n|\varepsilon_i|)$. Then we have
\begin{equation*}
\left(\begin{array}{cc}
1 & n \\
0 & 1
\end{array}\right)\left(\begin{array}{cc}
p & q \\
-b & a
\end{array}\right)(v_i, v_{i+1})=\left(\begin{array}{cc}
1 & -x_i \\
0 & |\varepsilon_i|
\end{array}\right).
\end{equation*}
Hence
\begin{equation*}
P_i=\left(\begin{array}{cc}
a & -q \\
b & p
\end{array}\right)\left(\begin{array}{cc}
1 & -n \\
0 & 1
\end{array}\right)
\end{equation*}
is a unimodular matrix.
When $\varepsilon_i<0$, we can show that the assertion holds by a similar argument.
Since $\left(\begin{array}{c}-x_i\\|\varepsilon_i|\end{array}\right)=P_i^{-1}v_{i+1}$ is primitive,
$x_i$ and $|\varepsilon_i|$ are relatively prime.
\end{proof}

Note that ${\rm det}(P_i)=\cfrac{\varepsilon_i}{|\varepsilon_i|}$.
Similarly, there exists a unique non-negative integer $y_i<|\varepsilon_i|$ such that
\begin{equation}\label{y_i}
Q_i=(v_{i+1}, v_i)\left(\begin{array}{cc}
1 & -y_i \\
0 & |\varepsilon_i|
\end{array}\right)^{-1}
\end{equation}
is a unimodular matrix.

Since $x_i$ and $|\varepsilon_i|$ are relatively prime,
$x_i>0$ when $|\varepsilon_i| \geq 2$ and $x_i=0$ when $|\varepsilon_i|=1$.
For $i$ such that $|\varepsilon_i| \geq 2$, let
\begin{equation}\label{l_i}
\frac{|\varepsilon_i|}{x_i}
=n^{(i)}_1-\cfrac{1}{n^{(i)}_2-\cfrac{1}{\ddots-\cfrac{1}{n^{(i)}_{l_i}}}}, \ 
n^{(i)}_j \geq 2.
\end{equation}
be the Hirzebruch-Jung continued fraction expansion.
This continued fraction expansion is unique.
We define $l_i=0$ when $|\varepsilon_i|=1$.

\begin{lemma}\label{a_i}
Let $a_i=\varepsilon_{i-1}^{-1}\varepsilon_i^{-1}{\rm det}(v_{i+1}, v_{i-1})$.
Then $a_i$ satisfies
\begin{equation}\label{a_i 2}
\varepsilon_{i-1}^{-1}v_{i-1}+\varepsilon_{i}^{-1}v_{i+1}+a_iv_i=0.
\end{equation}
\end{lemma}

\begin{proof}
It is easy to check that
\begin{equation*}
{\rm det}(v_i, v_{i+1})v_{i-1}+{\rm det}(v_{i-1}, v_i)v_{i+1}+{\rm det}(v_{i+1}, v_{i-1})v_i=0.
\end{equation*}
Dividing both sides by $\varepsilon_{i-1}\varepsilon_i={\rm det}(v_{i-1}, v_i){\rm det}(v_i, v_{i+1})$,
we obtain (\ref{a_i 2}).
\end{proof}

The following is our main theorem:

\begin{theorem}\label{main theorem}
Let $v_1, \ldots, v_d$ be a sequence of primitive vectors
and $\varepsilon_i={\rm det}(v_i, v_{i+1})$,
$a_i=\varepsilon_{i-1}^{-1}\varepsilon_i^{-1}{\rm det}(v_{i+1}, v_{i-1})$.
Let $x_i, y_i, l_i$, and $n^{(i)}_j$ be the integers defined in
Lemma \ref{x_i}, {\rm (\ref{y_i})}, and {\rm (\ref{l_i})}.
Then the rotation number of the sequence $v_1, \ldots, v_d$ around the origin is given by
\begin{equation}\label{formula}
\frac{1}{12}\sum_{i=1}^d\left(\left(3(l_i+1)
-\sum_{j=1}^{l_i}n^{(i)}_j\right)\frac{\varepsilon_i}{|\varepsilon_i|}
+a_i-\frac{x_i+y_i}{\varepsilon_i}\right).
\end{equation}
\end{theorem}

\begin{example}
Let $d=5$ and
\begin{equation*}
v_1=\left(\begin{array}{c}1\\0\end{array}\right),
v_2=\left(\begin{array}{c}1\\3\end{array}\right),
v_3=\left(\begin{array}{c}-2\\-1\end{array}\right),
v_4=\left(\begin{array}{c}-2\\1\end{array}\right),
v_5=\left(\begin{array}{c}5\\-3\end{array}\right).
\end{equation*}
Then we have the following:
\begin{center}
\begin{tabular}{|c|c|c|c|c|c|c|c|}\hline
$i$ & $\varepsilon_i$ & $a_i$ & $x_i$ & $y_i$ & $l_i$ & $n^{(i)}_1$ & $n^{(i)}_2$ \\ \hline
1 & 3 & $-2$ & 2 & 2 & 2 & 2 & 2 \\
2 & 5 & $\frac{1}{15}$ & 2 & 3 & 2 & 2 & 3 \\
3 & $-4$ & $\frac{7}{20}$ & 1 & 1 & 1 & 4 & \\
4 & 1 & $\frac{11}{4}$ & 0 & 0 & 0 & & \\
5 & 3 & $\frac{1}{3}$ & 1 & 1 & 1 & 3 & \\ \hline
\end{tabular}
\end{center}
\begin{figure}[htbp]
\begin{center}
\includegraphics{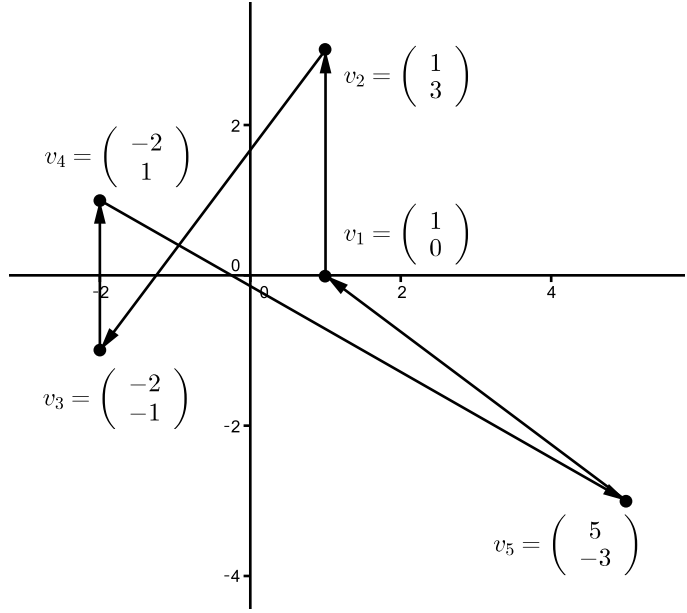}
\end{center}
\caption{A sequence of primitive vectors}
\label{fig:one}
\end{figure}
So we have
\begin{eqnarray*}
&&\sum_{i=1}^d\left(3(l_i+1)
-\sum_{j=1}^{l_i}n^{(i)}_j\right)\frac{\varepsilon_i}{|\varepsilon_i|} \\
&=&(3(2+1)-4)+(3(2+1)-5)-(3(1+1)-4)+3+(3(1+1)-3)=13, \\
&&\sum_{i=1}^da_i=-2+\cfrac{1}{15}+\cfrac{7}{20}+\cfrac{11}{4}+\cfrac{1}{3}=\cfrac{3}{2}, \\
&&\sum_{i=1}^d\frac{x_i+y_i}{\varepsilon_i}=\cfrac{2+2}{3}+\cfrac{2+3}{5}+\cfrac{1+1}{-4}+\cfrac{0+0}{1}+\cfrac{1+1}{3}=\cfrac{5}{2}.
\end{eqnarray*}
Therefore the value (\ref{formula}) is $\cfrac{1}{12}\left(13+\cfrac{3}{2}-\cfrac{5}{2}\right)=1$,
while the rotation number of the sequence $v_1, \ldots, v_5$ in Figure \ref{fig:one} is clearly one.
\end{example}

\section{Continued fractions}

Let $m \geq 2$ and $x(<m)$ be a positive integer prime to m, and let
\begin{equation*}
\frac{m}{x}
=n_1-\cfrac{1}{n_2-\cfrac{1}{\ddots-\cfrac{1}{n_l}}}, \ 
n_j \geq 2
\end{equation*}
be the continued fraction expansion.
This continued fraction expansion is unique,
and is called a {\it Hirzebruch-Jung continued fraction}.

\begin{lemma}\label{continued fraction 1}
Let $m \geq 2$ and $x(<m)$ be a positive integer prime to m, and let
\begin{equation*}
\frac{m}{x}
=n_1-\cfrac{1}{n_2-\cfrac{1}{\ddots-\cfrac{1}{n_l}}}, \ 
n_j \geq 2
\end{equation*}
be the continued fraction expansion.
Let $y(<m)$ be a unique positive integer such that $xy \equiv 1 \pmod{m}$.
Then the following identity holds:
\begin{equation*}
\left(\begin{array}{cc}
0 & -1 \\
1 & n_1
\end{array}\right)\left(\begin{array}{cc}
0 & -1 \\
1 & n_2
\end{array}\right)\ldots\left(\begin{array}{cc}
0 & -1 \\
1 & n_l
\end{array}\right)=\left(\begin{array}{cc}
\cfrac{1-xy}{m} & -x \\
y & m
\end{array}\right).
\end{equation*}
\end{lemma}

\begin{proof}
We prove this by induction on $l$.

If $l=1$, then we must have $x=1, n_1=m$ and $y=1$.
So the lemma holds when $l=1$.

Suppose that $l \geq 2$ and the lemma holds for $l-1$. We have
\begin{equation*}
\frac{x}{n_1x-m}
=n_2-\cfrac{1}{n_3-\cfrac{1}{\ddots-\cfrac{1}{n_l}}}.
\end{equation*}
Since $l \geq 2$, we have $x \geq 2$.
Since the right hand side in the identity above is greater than $1$, we have $0<n_1x-m<x$.
Since $m$ and $x$ are relatively prime, $x$ and $n_1x-m$ are relatively prime.
Moreover $\cfrac{xy-1}{m}$ is a positive integer less than $x$
and $(n_1x-m)\cfrac{xy-1}{m} \equiv 1 \pmod{x}$.
Hence by the hypothesis of induction, we obtain
\begin{eqnarray*}
\left(\begin{array}{cc}
0 & -1 \\
1 & n_2
\end{array}\right)\ldots\left(\begin{array}{cc}
0 & -1 \\
1 & n_l
\end{array}\right)&=&\left(\begin{array}{cc}
\cfrac{1-(n_1x-m)\cfrac{xy-1}{m}}{x} & -(n_1x-m) \\
\cfrac{xy-1}{m} & x
\end{array}\right) \\
&=&\left(\begin{array}{cc}
y-n_1\cfrac{xy-1}{m} & -(n_1x-m) \\
\cfrac{xy-1}{m} & x
\end{array}\right).
\end{eqnarray*}
Therefore we have
\begin{eqnarray*}
\left(\begin{array}{cc}
0 & -1 \\
1 & n_1
\end{array}\right)\ldots\left(\begin{array}{cc}
0 & -1 \\
1 & n_l
\end{array}\right)&=&\left(\begin{array}{cc}
0 & -1 \\
1 & n_1
\end{array}\right)\left(\begin{array}{cc}
y-n_1\cfrac{xy-1}{m} & -(n_1x-m) \\
\cfrac{xy-1}{m} & x
\end{array}\right) \\
&=&\left(\begin{array}{cc}
\cfrac{1-xy}{m} & -x \\
y & m
\end{array}\right),
\end{eqnarray*}
proving the lemma for $l$.
\end{proof}

\begin{prop}
The following identity holds:
\begin{equation*}
\cfrac{m}{y}=n_l-\cfrac{1}{n_{l-1}-\cfrac{1}{\ddots-\cfrac{1}{n_1}}}.
\end{equation*}
\end{prop}

\begin{proof}
Let $f:M_2(\mathbb{Z}) \rightarrow M_2(\mathbb{Z})$
be the antihomomorphism defined by
\begin{equation*}
f\left(\left(\begin{array}{cc}
a & b \\
c & d
\end{array}\right)\right)=\left(\begin{array}{cc}
a & -c \\
-b & d
\end{array}\right).
\end{equation*}
By Lemma \ref{continued fraction 1}, we have
\begin{eqnarray*}
\left(\begin{array}{cc}
0 & -1 \\
1 & n_l
\end{array}\right)\ldots\left(\begin{array}{cc}
0 & -1 \\
1 & n_1
\end{array}\right)&=&f\left(\left(\begin{array}{cc}
0 & -1 \\
1 & n_1
\end{array}\right)\ldots\left(\begin{array}{cc}
0 & -1 \\
1 & n_l
\end{array}\right)\right) \\
&=&f\left(\left(\begin{array}{cc}
\cfrac{1-xy}{m} & -x \\
y & m
\end{array}\right)\right)=\left(\begin{array}{cc}
\cfrac{1-xy}{m} & -y \\
x & m
\end{array}\right),
\end{eqnarray*}
proving the proposition.
\end{proof}

\begin{remark}
A similar assertion holds for regular continued fractions.
Let
\begin{equation*}
\frac{m}{x}
=n_1+\cfrac{1}{n_2+\cfrac{1}{\ddots+\cfrac{1}{n_l}}}, \ 
n_j \geq 1
\end{equation*}
be a continued fraction expansion,
and let $y(<m)$ be a unique positive integer such that $xy \equiv (-1)^{l+1} \pmod{m}$.
Then the following identity holds:
\begin{equation*}
\left(\begin{array}{cc}
0 & 1 \\
1 & n_1
\end{array}\right)\left(\begin{array}{cc}
0 & 1 \\
1 & n_2
\end{array}\right)\ldots\left(\begin{array}{cc}
0 & 1 \\
1 & n_l
\end{array}\right)=\left(\begin{array}{cc}
\cfrac{xy+(-1)^l}{m} & x \\
y & m
\end{array}\right).
\end{equation*}
The proof is similar to Lemma \ref{continued fraction 1}.
The following identity can be deduced by taking transpose at the identity above:
\begin{equation*}
\cfrac{m}{y}=n_l+\cfrac{1}{n_{l-1}+\cfrac{1}{\ddots+\cfrac{1}{n_1}}}.
\end{equation*}
\end{remark}

\section{Proof of Theorem \ref{main theorem}}

In this section, we give a proof of Theorem \ref{main theorem}.
We will use the notation in Section 2 freely.
We need the following lemma.

\begin{lemma}\label{xy}
For each $i=1, \ldots, d$, the following identity holds:
\begin{equation*}
\left(\begin{array}{cc}
0 & -1 \\
1 & n^{(i)}_1
\end{array}\right)\left(\begin{array}{cc}
0 & -1 \\
1 & n^{(i)}_2
\end{array}\right)\ldots\left(\begin{array}{cc}
0 & -1 \\
1 & n^{(i)}_{l_i}
\end{array}\right)=\left(\begin{array}{cc}
\cfrac{1-x_iy_i}{|\varepsilon_i|} & -x_i \\
y_i & |\varepsilon_i|
\end{array}\right).
\end{equation*}
\end{lemma}

\begin{proof}
If $|\varepsilon_i|=1$, then $x_i=y_i=l_i=0$
and the left hand side above is understood to be the identity matrix.
Assume $|\varepsilon_i| \geq 2$.
By Lemma \ref{x_i}, $x_i$ and $|\varepsilon_i|$ are relatively prime. Since
\begin{eqnarray*}
Q_i^{-1}P_i&=&\left(\begin{array}{cc}
1 & -y_i \\
0 & |\varepsilon_i|
\end{array}\right)(v_{i+1}, v_i)^{-1}(v_i, v_{i+1})\left(\begin{array}{cc}
1 & -x_i \\
0 & |\varepsilon_i|
\end{array}\right)^{-1} \\
&=&\frac{1}{|\varepsilon_i|}\left(\begin{array}{cc}
1 & -y_i \\
0 & |\varepsilon_i|
\end{array}\right)\left(\begin{array}{cc}
0 & 1 \\
1 & 0
\end{array}\right)\left(\begin{array}{cc}
|\varepsilon_i| & x_i \\
0 & 1
\end{array}\right)=\left(\begin{array}{cc}
-y_i & \cfrac{1-x_iy_i}{|\varepsilon_i|} \\
|\varepsilon_i| & x_i
\end{array}\right)
\end{eqnarray*}
is a unimodular matrix, $x_iy_i$ is congruent to $1$ modulo $|\varepsilon_i|$.
Therefore the lemma follows from Lemma \ref{continued fraction 1}.
\end{proof}

\begin{proof}[Proof of Theorem \ref{main theorem}]
For $j=0, \ldots, l_i+1$, we define
\begin{equation}\label{w}
w^{(i)}_j=\left\{\begin{array}{ll}
P_i\left(\begin{array}{cc}
0 & -1 \\
1 & n^{(i)}_1
\end{array}\right)\cdots\left(\begin{array}{cc}
0 & -1 \\
1 & n^{(i)}_j
\end{array}\right)\left(\begin{array}{c}
1 \\
0
\end{array}\right) & (0 \leq j \leq l_i), \\
P_i\left(\begin{array}{cc}
0 & -1 \\
1 & n^{(i)}_1
\end{array}\right)\cdots\left(\begin{array}{cc}
0 & -1 \\
1 & n^{(i)}_{j-1}
\end{array}\right)\left(\begin{array}{c}
0 \\
1
\end{array}\right) & (1 \leq j \leq l_i+1).
\end{array}\right.
\end{equation}
Note that both expressions at the right hand side of (\ref{w})
are equal if $1 \leq j \leq l_i$.
By the definition of $w^{(i)}_j$, it follows that
\begin{equation}\label{epsilon}
{\rm det}(w^{(i)}_j,w^{(i)}_{j+1})
={\rm det}(P_i){\rm det}\left(\left(\begin{array}{c}
1 \\
0
\end{array}\right), \left(\begin{array}{c}
0 \\
1
\end{array}\right)\right)=\cfrac{\varepsilon_i}{|\varepsilon_i|}
\in \{\pm1\}
\end{equation}
for any $j=0, \ldots, l_i$. So the sequence
\begin{equation}\label{unimodular}
\ldots, v_i=w^{(i)}_0, w^{(i)}_1, \ldots, w^{(i)}_{l_i+1}=v_{i+1}, \ldots
\end{equation}
is unimodular.

\begin{figure}[htbp]
\begin{center}
\includegraphics{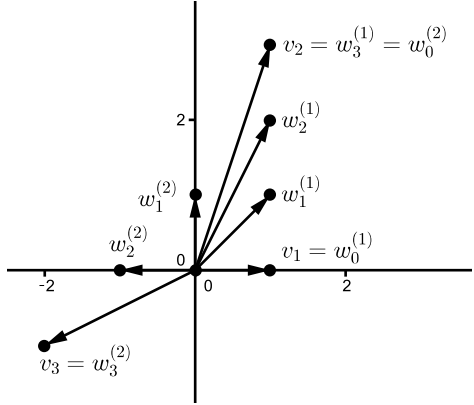}
\end{center}
\caption{Adding $w^{(i)}_j$ to the given vector sequence}
\label{fig:two}
\end{figure}

Hence by Theorem \ref{Higashitani-Masuda},
the rotation number of $v_1, \ldots, v_d$ is given by
\begin{eqnarray}\label{rotation}
&& \frac{1}{4}\sum_{i=1}^d\sum_{j=0}^{l_i}{\rm det}(w^{(i)}_j, w^{(i)}_{j+1}) \nonumber \\
&+& \frac{1}{12}\sum_{i=1}^d\frac{{\rm det}(w^{(i)}_1, w^{(i-1)}_{l_{i-1}})}
{{\rm det}(w^{(i-1)}_{l_{i-1}}, v_i){\rm det}(v_i, w^{(i)}_1)} \\
&+& \frac{1}{12}\sum_{i=1}^d\sum_{j=1}^{l_i}\frac{{\rm det}(w^{(i)}_{j+1}, w^{(i)}_{j-1})}
{{\rm det}(w^{(i)}_{j-1}, w^{(i)}_j){\rm det}(w^{(i)}_j,w^{(i)}_{j+1})}. \nonumber
\end{eqnarray}

As for the first summand in (\ref{rotation}), it follows from (\ref{epsilon}) that
\begin{equation*}
\sum_{j=0}^{l_i}{\rm det}(w^{(i)}_j, w^{(i)}_{j+1})
=(l_i+1)\frac{\varepsilon_i}{|\varepsilon_i|}.
\end{equation*}

As for the second summand in (\ref{rotation}), we first observe that
it follows from Lemma \ref{a_i} that
\begin{eqnarray*}
P_{i-1}^{-1}P_i&=&\left(\begin{array}{cc}
1 & -x_{i-1} \\
0 & |\varepsilon_{i-1}|
\end{array}\right)(v_{i-1}, v_i)^{-1}(v_i, v_{i+1})\left(\begin{array}{cc}
1 & -x_i \\
0 & |\varepsilon_i|
\end{array}\right)^{-1} \\
&=&\frac{1}{|\varepsilon_i|}\left(\begin{array}{cc}
1 & -x_{i-1} \\
0 & |\varepsilon_{i-1}|
\end{array}\right)(v_{i-1}, v_i)^{-1}
(v_i, -\varepsilon_i(\varepsilon_{i-1}^{-1}v_{i-1}+a_iv_i))\left(\begin{array}{cc}
|\varepsilon_i| & x_i \\
0 & 1
\end{array}\right) \\
&=&\frac{1}{|\varepsilon_i|}\left(\begin{array}{cc}
1 & -x_{i-1} \\
0 & |\varepsilon_{i-1}|
\end{array}\right)\left(\begin{array}{cc}
0 & -\varepsilon_i\varepsilon_{i-1}^{-1} \\
1 & -a_i\varepsilon_i
\end{array}\right)\left(\begin{array}{cc}
|\varepsilon_i| & x_i \\
0 & 1
\end{array}\right) \\
&=&\frac{1}{|\varepsilon_i|}\left(\begin{array}{cc}
-|\varepsilon_i|x_{i-1} & -\varepsilon_i\varepsilon_{i-1}^{-1}-x_{i-1}x_i+a_i\varepsilon_ix_{i-1} \\
|\varepsilon_{i-1}||\varepsilon_i| & |\varepsilon_{i-1}|(x_i-a_i\varepsilon_i)
\end{array}\right).
\end{eqnarray*}
So it follows from (\ref{unimodular}), (\ref{w}), (\ref{epsilon}), and Lemma \ref{xy} that
\begin{eqnarray*}
&& \frac{{\rm det}(w^{(i)}_1, w^{(i-1)}_{l_{i-1}})}
{{\rm det}(w^{(i-1)}_{l_{i-1}}, v_i){\rm det}(v_i, w^{(i)}_1)}
=\frac{{\rm det}(w^{(i)}_1, w^{(i-1)}_{l_{i-1}})}
{{\rm det}(w^{(i-1)}_{l_{i-1}}, w^{(i-1)}_{l_{i-1}+1}){\rm det}(w^{(i)}_0, w^{(i)}_1)} \\
&=& \frac{|\varepsilon_{i-1}||\varepsilon_i|}{\varepsilon_{i-1}\varepsilon_i}
{\rm det}\left(P_i\left(\begin{array}{c}
0 \\
1
\end{array}\right), P_{i-1}\left(\begin{array}{cc}
0 & -1 \\
1 & n^{(i-1)}_1
\end{array}\right)\cdots\left(\begin{array}{cc}
0 & -1 \\
1 & n^{(i-1)}_{l_{i-1}}
\end{array}\right)\left(\begin{array}{c}
1 \\
0
\end{array}\right)\right) \\
&=& \frac{|\varepsilon_{i-1}||\varepsilon_i|}{\varepsilon_{i-1}\varepsilon_i}
{\rm det}(P_{i-1}){\rm det}\left(P_{i-1}^{-1}P_i\left(\begin{array}{c}
0 \\
1
\end{array}\right), \left(\begin{array}{c}
\cfrac{1-x_{i-1}y_{i-1}}{|\varepsilon_{i-1}|} \\
y_{i-1}
\end{array}\right)\right) \\
&=& \frac{1}{\varepsilon_i}
{\rm det}\left(\begin{array}{cc}
-\cfrac{\varepsilon_i}{\varepsilon_{i-1}}-x_{i-1}x_i+a_i\varepsilon_ix_{i-1} &
\cfrac{1-x_{i-1}y_{i-1}}{|\varepsilon_{i-1}|} \\
|\varepsilon_{i-1}|(x_i-a_i\varepsilon_i) & y_{i-1}
\end{array}\right) \\
&=&a_i-\frac{x_i}{\varepsilon_i}-\frac{y_{i-1}}{\varepsilon_{i-1}}.
\end{eqnarray*}

As for the last summand in (\ref{rotation}), it follows from (\ref{w}) and (\ref{epsilon}) that
\begin{eqnarray*}
&& \frac{{\rm det}(w^{(i)}_{j+1}, w^{(i)}_{j-1})}
{{\rm det}(w^{(i)}_{j-1}, w^{(i)}_j){\rm det}(w^{(i)}_j,w^{(i)}_{j+1})} \\
&=& {\scriptstyle {\rm det}(P_i){\rm det}\left(\left(\begin{array}{cc}
{\scriptstyle 0} & {\scriptstyle -1} \\
{\scriptstyle 1} & {\scriptstyle n^{(i)}_1}
\end{array}\right)\cdots\left(\begin{array}{cc}
{\scriptstyle 0} & {\scriptstyle -1} \\
{\scriptstyle 1} & {\scriptstyle n^{(i)}_j}
\end{array}\right)\left(\begin{array}{c}
{\scriptstyle 0} \\
{\scriptstyle 1}
\end{array}\right), \left(\begin{array}{cc}
{\scriptstyle 0} & {\scriptstyle -1} \\
{\scriptstyle 1} & {\scriptstyle n^{(i)}_1}
\end{array}\right)\cdots\left(\begin{array}{cc}
{\scriptstyle 0} & {\scriptstyle -1} \\
{\scriptstyle 1} & {\scriptstyle n^{(i)}_{j-1}}
\end{array}\right)\left(\begin{array}{c}
{\scriptstyle 1} \\
{\scriptstyle 0}
\end{array}\right)\right)} \\
&=& {\rm det}(P_i){\rm det}\left(\left(\begin{array}{cc}
0 & -1 \\
1 & n^{(i)}_j
\end{array}\right)\left(\begin{array}{c}
0 \\
1
\end{array}\right), \left(\begin{array}{c}
1 \\
0
\end{array}\right)\right) \\
&=& \frac{\varepsilon_i}{|\varepsilon_i|}{\rm det}\left(\begin{array}{cc}
-1 & 1 \\
n^{(i)}_j & 0
\end{array}\right)=-n^{(i)}_j\frac{\varepsilon_i}{|\varepsilon_i|}.
\end{eqnarray*}

Therefore (\ref{rotation}) reduces to
\begin{eqnarray*}
&& \frac{1}{4}\sum_{i=1}^d(l_i+1)\frac{\varepsilon_i}{|\varepsilon_i|}
+\frac{1}{12}\sum_{i=1}^d\left(a_i-\frac{x_i}{\varepsilon_i}-\frac{y_{i-1}}{\varepsilon_{i-1}}\right)
+\frac{1}{12}\sum_{i=1}^d\sum_{j=1}^{l_i}\left(-n^{(i)}_j\frac{\varepsilon_i}{|\varepsilon_i|}\right) \\&=& \frac{1}{12}\sum_{i=1}^d\left(\left(3(l_i+1)
-\sum_{j=1}^{l_i}n^{(i)}_j\right)\frac{\varepsilon_i}{|\varepsilon_i|}
+a_i-\frac{x_i+y_i}{\varepsilon_i}\right),
\end{eqnarray*}
proving the theorem.
\end{proof}

\begin{remark}
It sometimes happens that a sequence of primitive vectors $v_1, \ldots, v_d$ is not unimodular
but is unimodular with respect to the sublattice of $\mathbb{Z}^2$ generated by
vectors $v_1, \ldots, v_d$. Such a sequence is called an $l$-reflexive loop
and studied in \cite{Kasprzyk and Nill}.
Theorem \ref{Higashitani-Masuda} can be applied to an $l$-reflexive loop
with respect to the sublattice generated by the vectors in the $l$-reflexive loop,
but it is unclear whether the resulting formula can be obtained from Theorem \ref{main theorem}.
\end{remark}

\section{Acknowledgement}

The author wishes to thank Professor Mikiya Masuda
for his valuable advice and continuing support.

\end{document}